\documentclass[reqno]{amsart}
\usepackage{amsmath}
\usepackage{amssymb}
\usepackage{amsthm}
\usepackage{mathrsfs}
\allowdisplaybreaks[4]

\makeatletter
 
  \@addtoreset{equation}{section}
 \makeatother

\def \ds{\displaystyle}
\def \wt{\widetilde}

\def \R{\mathbb{R}}
\def \C{\mathbb{C}}

\def \N{\mathbb{N}}
\def \v{\varphi}
\def \p{\rho}
\def \f{\phi}

\def \e{\varepsilon}
\def \l{\lambda}

\def \D{\Delta}

\def \pa{\partial}
\def \n{\nabla}

\def \Si{\Sigma}
\def \a{\alpha}

\def \n{\nabla}
\def \t{\theta}

\def \ta{\tau}

\def \ga{\gamma}

\newcommand{\im}{\operatorname{Im}}
\newcommand{\re}{\operatorname{Re}}
\theoremstyle{definition}
\newtheorem{dei}{Definition}[section]
\newtheorem{thm}{Theorem}[section]
\newtheorem{pro}{Proposition}[section]
\newtheorem{lem}{Lemma}[section]

\newtheorem{rem}{Remark}[section]

\newtheorem*{nota}{Notation}

\newcommand{\todayd}{\the\year/\the\month/\the\day}

\begin{document}
\title[The derivation of the conservation laws]{The derivation of  conservation laws for nonlinear Schr{\"o}dinger equations with a power type nonlinearity}
\author[K. Fujiwara]{Kazumasa Fujiwara}
\address[]{Department of Pure and Applied Physics, Waseda University, Tokyo 169-8555, JAPAN}
\email{k-fujiwara@asagi.waseda.jp}
\author[H. Miyazaki]{Hayato MIYAZAKI}
\address[]{Department of Mathematics, Graduate School of Science, Hiroshima University, Higashi-Hiroshima, 739-8521, JAPAN}
\email{h-miyazaki@hiroshima-u.ac.jp}

\begin{abstract}
For nonlinear Schr{\"o}dinger equations with a power nonlinearity, a new approach to derive the conservation law of the momentum and the pseudo conformal conservation law is obtained. Since this approach does not contain approximating procedure, the argument is simplified to derive these conservation laws.   
\end{abstract}

\maketitle 

\section{Introduction}
In this paper, we consider the Cauchy problem of nonlinear Schr{\"o}dinger equations with a power nonlinearity. 
\begin{align}
\left\{
\begin{array}{l}
\ds i\pa_{t}u+\frac{1}{2}\Delta u= \l|u|^{p-1}u,\quad t\in(0,T),\;x\in\R^{n}, \\
u(0,x)=\f(x),\quad x\in\R^{n},
\end{array}
\right. \label{nls}
\end{align}
where $u(t,x): (0,T)\times \R^{n} \rightarrow \C$, $\l \in \R$, $p>1$, and the initial data $\f$ is a complex valued function on $\R^{n}$. 
The equation (\ref{nls}) has been extensively studied both in the physical and mathematical literatures (see \cite{bib007}, \cite{bib001}). 
The conservation law of the momentum and the pseudo conformal conservation law play a role to investigate a asymptotic behavior of the solution of (\ref{nls}). For example, using the conservation law of the momentum, we study blow-up in finite time and dynamics of blow-up solutions (see \cite{bib006}). Moreover, the pseudo conformal conservation law implies that if $p \geq 1+ 4/n$, $\l>0$, then for any $r \in [2, \a(n)]$, we have a time decay estimate for the solution $u \in C([0,T], \Si )$ of (\ref{nls}) such that $\|u(t)\|_{L^{r}} \leq C|t|^{-n(\frac{1}{2}-\frac{1}{r})}$, where $\a(n)= 1+ 4/(n-2)$ if $n \geq 3$, $\a(n) =\infty$ if $n=1$, $2$  and $\Si = \{u \in H^{1}(\R^n)\ |\ xu \in L^{2}(\R^{n})\}$(see \S 7 of \cite{bib001}). 

We can obtain formally the conservation law of the moumentum $P(u)$ by multiplying the equation (\ref{nls}) by $\n \bar{u}$, integrating over $\R^{n}$, and taking the real part as follows:
\begin{align*}
0 &= 2\re \left(i\pa_{t}u+\frac{1}{2}\D u-\l|u|^{p-1}u, \n u \right)_{L^{2}} \\
&= 2\re (i\pa_{t}u, \n u)_{L^{2}} - 2\re (\l|u|^{p-1}u, \n u)_{L^{2}} = \frac{d}{dt}P(u(t)),
\end{align*}
where
\begin{align*}
P(u) = \im \int_{\R^{n}} u \n \overline{u} dx.
\end{align*}

We present two formal arguments to derive the pseudo conformal conservation law. One is that by using the virial identity, we derive it (see Theorem 7.2.1 of \cite{bib001}). Other is that applying a transform 
\begin{align}
u(t, x) = (it)^{-\frac{n}{2}}e^{\frac{i|x|^{2}}{2t}}\overline{v\left(\frac{1}{t}, \frac{x}{t}\right)} \label{tra:01}
\end{align}
to (\ref{nls}), we have a equation
\begin{align}
 i\pa_{t}v+\frac{1}{2}\Delta v= \l t^{\frac{n(p-1)-4}{2}}|v|^{p-1}v. \label{eq:11}
\end{align} 
Furthermore, by multiplying the equation (\ref{eq:11}) by $-\overline{\pa_{t}v}$, integrating over $\R^{n}$, and taking the real part, we deduce that 
\begin{align}
0 &=2\re \left(i\pa_{t}v+\frac{1}{2}\D v - \l t^{\frac{n(p-1)-4}{2}}|v|^{p-1}v, -\pa_t v \right)_{L^{2}} \nonumber \\
&= -\re \left( \D v , \pa_t v \right)_{L^{2}} +2\re \left( \l t^{\frac{n(p-1)-4}{2}}|v|^{p-1}v, \pa_t v \right)_{L^{2}} \nonumber \\
&= \frac{d}{dt}E_{1}(v(t)) - \{n(p-1)-4\}t^{\frac{n(p-1)-6}{2}}\frac{\l}{p+1}\|v(t)\|_{L^{p+1}}^{p+1}, \label{ps:01}
\end{align}
where
\[
E_{1}(v) = \frac{1}{2}\|\n v(u)\|_{L^{2}}^{2} + t^{\frac{n(p-1)-4}{2}}\frac{2\l}{p+1}\|v\|_{L^{p+1}}^{p+1}.
\]
Finally, since 
\begin{align}
\|\n v(t)\|_{L^{2}} = \|(y+is\n)u(s)\|_{L^{2}}, \quad  \|v(t)\|_{L^{p+1}} = s^{\frac{n(p-1)}{2(p+1)}}\|u(s)\|_{L^{p+1}} \label{tran:01} 
\end{align}
with $t=1/s$ and $y=x/t$ (see \S 7.5 of \cite{bib001}), the equality (\ref{ps:01}) yields the pseudo conformal conservation law
\begin{align}
&\frac{1}{2}\|(x+it \n) u(t)\|_{L^{2}}^{2} + t^{2}\frac{2\l}{p+1}\|u(t)\|_{L^{p+1}}^{p+1} \nonumber \\
&\qquad = \frac{1}{2}\|x \f\|_{L^{2}}^{2} - \frac{\l \{n(p-1)-4\}}{p+1}\int_{0}^{t}\ta \|u(\ta)\|_{L^{p+1}}^{p+1} d\ta \label{ps:02}
\end{align}
(see \S 7.5 of \cite{bib001}). To justify the procedures above, we require that at least $u$ and $v$ are $H^{2}$-solutions. For an $H^{s}$-solution with $s<2$, there are basically two methods to justify the procedure. One is that solutions are approximated by a sequence of regular solutions, using the continuous dependence of solutions on the initial data. Other is to use a sequence of regularized equations of (\ref{nls}) whose solutions have enough regularities to perform the procedure above (see \cite{bib002}). However, these two methods involve a limiting procedure on approximate solutions. Instead, for (\ref{nls}), Ozawa \cite{bib002} derive conservation laws of the charge and the energy by using additional properties of solutions provided by Strichartz estimates. We need the following notations and definitions to mention it:

\begin{nota}
For a Banach space $X$, $p \in [1, \infty]$ and an interval $I \subset \R$, $L_{t}^{p}X$ denotes the Banach space $L^{p}(I, X)$ equipped with its natural norm.
Let $U(t)$ be the Schr{\"o}dinger operator $e^{\frac{it}{2}\D}$. We denote by $f(u)$ the nonlinearity $\l|u|^{p-1}u$.
For $s \in \R$ and $p, q \in [1, \infty]$, Let $B^{s}_{p, q}(\R^{n})=B^{s}_{p, q}$ be the inhomogeneous Besov space in $\R^{n}$.
\end{nota}

\begin{dei}
\begin{enumerate}
\item A positive exponent $p'$ is called the dual exponent of $p$ if $p$ and $p'$ satisfy $1/p+1/p'=1$. 
\item A pair of two exponents $(p,q)$ is called an admissible pair if $(p.q)$ satisfies $2/p+n/q=n/2$, $p \geq 2$ and $(p,q) \neq (2,\infty)$. 
\end{enumerate}
\end{dei}
Strichartz estimates are described as the following lemma:
\begin{lem}[Strichartz estimates, see \cite{bib001}]
Let $s \in \R$ and $I \subset \R$ be an interval with $0 \in \bar{I}$. $(p_{1},q_{1})$ and $(p_{2},q_{2})$ denote admissible pairs. Let $t_{0} \in \bar{I}$. Then
\begin{enumerate}
\item for all $f\in L^{2}(\R^{n})$,
\[
\|U(t)f\|_{L^{p_{1}}(\R, L^{q_{1}}(\R^{n}))} \leq C\|f\|_{L^{2}(\R^{n})},
\]
\item for any $\v\in H^{s}(\R^{n})$,
\[
\|U(t)\v\|_{L^{\infty}(\R, H^{s}(\R^{n}))} \leq C\|\v\|_{H^{s}(\R^{n})},
\]
\[
\|U(t)\v\|_{L^{p_{1}}(\R, B^{s}_{q_{1}, 2}(\R^{n}))} \leq C\|\v\|_{H^{s}(\R^{n})},
\]
\item for all $f\in L^{p_{1}'}(I,L^{q_{1}'}(\R^{n}))$,
\[
\left\|\int_{t_{0}}^{t}U(t-\ta)f(\ta)d\ta \right\|_{L^{p_{2}}(I,L^{q_{2}}(\R^{n}))} \leq C\|f\|_{L^{p'_{1}}(I,L^{q'_{1}}(\R^{n}))},
\]
\item for any $f\in L^{1}(I,H^{s}(\R^{n}))$,
\[
\left\|\int_{t_{0}}^{t}U(t-\ta)f(\ta)d\ta \right\|_{L^{\infty}(I,H^{s}(\R^{n}))} \leq C\|f\|_{L^{1}(I,H^{s}(\R^{n}))},
\]
\item for all $f\in L^{p_{1}'}(I, B^{s}_{q_{1}', 2}(\R^{n}))$,
\[
\left\|\int_{t_{0}}^{t}U(t-\ta)f(\ta)d\ta \right\|_{L^{p_{2}}(I,B^{s}_{q_{2}, 2}(\R^{n}))} \leq C\|f\|_{L^{p'_{1}}(I,B^{s}_{q_{1}', 2}(\R^{n}))},
\]
\end{enumerate}
where $p'_{1}$ and $p'_{2}$ are the dual exponents of $p_{1}$ and $p_{2}$, respectively.
\end{lem}

\section{Main results}

In this paper, for the equation (\ref{nls}), we derive the conservation law of the momentum and the pseudo conformal conservation law for time local solutions without approximating procedure. Instead of that, we use Ozawa's idea \cite{bib002}. For the conservation law of the momentum, our first result is as follows:
\begin{pro}\label{pro:01}
Let $1/2 \leq s<\min\{1, n/2\}$, $p>1$ and $\l \in \R$. Let an admissible pair $(\ga, \p)$ be as follows:
\begin{align}
\p = \frac{n(p+1)}{n+s(p-1)}, \quad \ga= \frac{4(p+1)}{(p-1)(n-2s)}. \label{ad:01}
\end{align}
Let $u \in C ([0, T], H^{s}(\R^{n})) \cap L^{\ga}((0, T), B^{s}_{\p, 2}(\R^{n}))$ be a mild solution of an integral equaiton
\begin{align}
u(t) = U(t)\f - i \int_{0}^{t}U(t-\ta)f(u(\ta))d\ta \label{ieq:01}
\end{align}
for some $\f \in H^{s}$ and $T>0$. Then $P(u(t))=P(\f)$ for all $t \in [0, T]$, where
\[
P(u):= \im \int_{\R^{n}} u \n \overline{u} dx. 
\] 
\end{pro}

\begin{rem} \label{rem:21}
Cazenave-Weissler \cite{bib003} prove that if $0 < s<\min\{1, n/2\}$, $\f \in H^{s}$ and $1 \leq p < 1+ 4/(n-2s)$, then the Cauchy problem (\ref{nls}) have an unique solution $u \in C ([0, T], H^{s}(\R^{n})) \cap L^{\ga}((0, T), B^{s}_{\p, 2}(\R^{n}))$ with some admissile pair $(\ga, \p)$ as in (\ref{ad:01}). We remark that $u \in L^{\ga}_{t}B^{s}_{\p, 2}$ implies $f(u) \in L^{\ga'}_{t} B^{s}_{\p',2}$.  
\end{rem}
\begin{rem}
Proposition \ref{pro:01} holds for more general nonlinearities satisfying the following conditions (see \S 4.9 of \cite{bib001}): 
\begin{enumerate}
\item[(A1)] $f \in C(\C, \C)$, $f(0)=0$ and there exists $p \in [1, 1+4/(n-2s))$ such that $f$ satisfies 
\begin{align*}
|f(z_{1})-f(z_{2})| \leq C(1+|z_{1}|^{p-1}+|z_{2}|^{p-1})|z_{1}-z_{2}| 
\end{align*}
for any $z_{1}$, $z_{2} \in \C$.
\item[(A2)] for all $\ta \in \R$, $f(\ta) \in \R$, and for any $z \in \C$ and $\t \in \R$, $f(e^{i \t}z) = e^{i \t}f(z)$.
\end{enumerate}
Note that for the nonlinearity $f$ satisfying (A1) and (A2), we deduce that $f(u) \in L^{\infty}_{t}H^{s} + L^{\ga'}_{t} B^{s}_{\p',2}$ if $u \in C ([0, T], H^{s}(\R^{n})) \cap L^{\ga}((0, T), B^{s}_{\p, 2}(\R^{n}))$.
\end{rem}
Second aim of this paper is to derive the pseudo conformal conservation law without approximating procedure, using Ozawa's idea \cite{bib002}. The locally well-poseness of (\ref{nls}) in $\Si$ is well known as follows: 
\begin{thm}(see \cite{bib001}) \label{thm:01}
Let $\l \in \R$ and $\a(n)= 1+ 4/(n-2)$ if $n \geq 3$, $\a(n) =\infty$ if $n=1$, $2$. Assume that $1<p<\a(n)$. for any $\f \in H^{1}$, there exists $T>0$ such that (\ref{nls}) has a unique solution $u \in C([0, T], H^{1})\cap L^{q}([0,T], W^{1,r})$, where a pair $(q,r)$ is some admissible pair.  Furthermore, if $\f \in \Si$, then $u \in C([0,T], \Si)$.
\end{thm}
For any $\wt{T} \in (0, T)$, the solution $u$ as in Theorem \ref{thm:01} satisfies the following intergral equation:  
\begin{align}
u(t)=U(t-\wt{T})u(\wt{T})- i \int_{\wt{T}}^{t}U(t-s)f(u(s))ds \label{eq:02}
\end{align}
for all $t \in [\wt{T},T]$. Using the transform (\ref{tra:01}), for any $\wt{T} \in (0, T)$, it follows from (\ref{eq:02}) that $v$ is a solution of the following intergral equation:  
\begin{align}
v(t)=U \left(t-\frac{1}{\wt{T}}\right)v\left(\frac{1}{\wt{T}}\right)- i \int_{\frac{1}{\wt{T}}}^{t}s^{\frac{n(p-1)-4}{2}}U(t-s)f(v(s))ds \label{eq:03}
\end{align}
with $v \in C\left(\left[1/T, 1/\wt{T} \right], \Si \right)$ (see Lemma 2.8 of \cite{bib008}). Note that the intgral equation (\ref{eq:03}) equivalents to the equation (\ref{eq:11}). The second result in this paper is as follows:
\begin{pro}\label{pro:02}
Let $p>1$, $\l \in \R$. A pair $(q,r)$ denotes some admissible pair. Let $v \in C([T_{1}, T_{2}], H^{1}) \cap L^{q}([T_{1}, T_{2}], W^{1,r})$ be a mild solution of the integral equation
\begin{align}
v(t)=U(t-T_{1})\f- i \int_{T_{1}}^{t}s^{\frac{n(p-1)-4}{2}}U(t-s)f(v(s))ds \label{eq:04}
\end{align}
for some $\f \in H^{1}$ and $T_{1}$, $T_{2}>0$. Then 
\begin{align*}
&\frac{1}{2}\|\n v(t)\|_{L^{2}}^{2} + \frac{2\l t^{\frac{n(p-1)-4}{2}}}{p+1}\|v(t)\|_{L^{p+1}}^{p+1} \\
&= \frac{1}{2}\|\n \f\|_{L^{2}}^{2} +\frac{2\l T_{1}^{\frac{n(p-1)-4}{2}}}{p+1}\|v(T_{1})\|_{L^{p+1}}^{p+1} \\
&\hspace{3cm} + \frac{\l \{n(p-1)-4\}}{p+1}\int_{T_{1}}^{t}s^{\frac{n(p-1)-6}{2}}\|v(s)\|_{L^{p+1}}^{p+1} ds
\end{align*}
for all $t \in [T_{1}, T_{2}]$.
\end{pro}

\begin{rem}
Applying Proposition \ref{pro:02} to the integral equation (\ref{eq:03}), we see that
\begin{align}
&\frac{1}{2}\|\n v(t)\|_{L^{2}}^{2} + 2t^{\frac{n(p-1)-4}{2}}\frac{\l}{p+1}\|v(t)\|_{L^{p+1}}^{p+1} \nonumber \\
&= \frac{1}{2} \left\|\n v\left(\frac{1}{\wt{T}}\right) \right\| + 2\left(\frac{1}{\wt{T}}\right)^{\frac{n(p-1)-4}{2}}\frac{\l}{p+1}\left\|v\left(\frac{1}{\wt{T}}\right)\right\|_{L^{p+1}}^{p+1} \nonumber \\
&\hspace{3cm} + \frac{\l \{n(p-1)-4\}}{p+1}\int_{\frac{1}{\wt{T}}}^{t}s^{\frac{n(p-1)-6}{2}}\|v(s)\|_{L^{p+1}}^{p+1} ds. \label{cons:01}
\end{align}
Combining (\ref{tran:01}) with (\ref{cons:01}), we get
\begin{align*}
&\frac{1}{2} \|(y+is \n) u(s)\|_{L^{2}}^{2} + \frac{2\l s^{2}}{p+1}\|u(s)\|_{L^{p+1}}^{p+1} \\
&=\frac{1}{2} \|(y+i\wt{T}\n)u(s)\|_{L^{2}}^{2}+ \frac{2\l \wt{T}^{2}}{p+1}\|u(\wt{T})\|_{L^{p+1}}^{p+1} \\
&\hspace{3cm} - \frac{\l \{n(p-1)-4\}}{p+1}\int_{\wt{T}}^{s}\ta \|u(\ta)\|_{L^{p+1}}^{p+1} d\ta
\end{align*}
for all $s \in [\wt{T}, T]$. Finally, it follows from $u \in C([0,T], \Si)$ that if $\wt{T} \rightarrow 0$, then we get the pseudo conformal conservation law (\ref{ps:02}).
\end{rem}

\section{The proof of main results}

\begin{proof}[Proof of Proposition \ref{pro:01}]
For all $t \in [0,T]$, we obtain
\begin{align*}
P(u(t)) &=\left\langle u, \overline{\n u} \right\rangle_{H^{1/2} \times H^{-1/2}} \\
&= \left\langle U(-t)u, \overline{U(-t)\n u} \right\rangle_{H^{1/2} \times H^{-1/2}} \\
&= \left\langle \f, \overline{\n \f} \right\rangle_{H^{1/2} \times H^{-1/2}} \\
&\quad + \left\langle \f, \overline{-i \int_{0}^{t}U(-\ta)\n f(u(\ta))d\ta} \right\rangle_{H^{1/2} \times H^{-1/2}} \\
&\quad +\left\langle -i \int_{0}^{t}U(-\wt{\ta})f(u(\wt{\ta}))d\wt{\ta} , \overline{\n \f} \right\rangle_{H^{1/2} \times H^{-1/2}} \\
&\quad + \left\langle -i \int_{0}^{t}U(-\wt{\ta})f(u(\wt{\ta}))d\wt{\ta}, \overline{-i \int_{0}^{t}U(-\ta)\n f(u(\ta))d\ta} \right\rangle_{H^{1/2} \times H^{-1/2}} \\
&= P(\f) \\
&\quad - \int_{0}^{t} \left\langle U(\ta)\f, \overline{i \n f(u(\ta))} \right\rangle d\ta - \int_{0}^{t} \left\langle i f(u(\wt{\ta})), \overline{U(\wt{\ta}) \n \f} \right\rangle d\wt{\ta} \\
&\quad - \int_{0}^{t} \left\langle i f(u(\wt{\ta})), \overline{-i \int_{0}^{\wt{\ta}} U(\wt{\ta}-\ta)\n f(u(\ta)) d\ta}  \right\rangle d\wt{\ta} \\
&\quad - \int_{0}^{t} \left\langle -i \int_{0}^{\ta} U(\ta-\wt{\ta}) f(u(\wt{\ta})) d\wt{\ta} , \overline{i \n f(u(\ta))} \right\rangle d\ta,
\end{align*}
where concatenating Strichartz estimates and Remark \ref{rem:21}, time integrals of the scalar product of each term of the last line in the above are understood as the duality coupling on $(L^{\ga}_{t}B_{\p, 2}^{1/2}) \times (L^{\ga'}_{t}B_{\p', 2}^{-1/2})$, $(L^{\ga'}_{t}B_{\p', 2}^{1/2}) \times (L^{\ga}_{t}B_{\p, 2}^{-1/2})$, $(L^{\ga'}_{t}B_{\p', 2}^{1/2}) \times (L^{\ga}_{t}B_{\p, 2}^{-1/2})$ and 
$(L^{\ga}_{t}B_{\p, 2}^{1/2}) \times (L^{\ga'}_{t}B_{\p', 2}^{-1/2})$, respectively. Using the integral equation (\ref{ieq:01}), we compute
\begin{align*}
P(u(t))&= P(\f) \\
&\quad - \int_{0}^{t} \left\langle u(\ta), \overline{i \n f(u(\ta))} \right\rangle d\ta - \int_{0}^{t} \left\langle i f(u(\wt{\ta})), \overline{\n u(\wt{\ta})} \right\rangle d\wt{\ta} \\
&= P(u_{0}) + 2 \int_{0}^{t} \re \left\langle f(u(\wt{\ta})), \overline{\n u(\wt{\ta})} \right\rangle d\wt{\ta}.
\end{align*}
Note that the time integral of the scalar product of the last line in the above is understood as the duality coupling on $(L^{\ga'}_{t}B_{\p', 2}^{1/2}) \times (L^{\ga}_{t}B_{\p, 2}^{-1/2})$. Hence, to continue the proof, we need the following elementary lemma:
\begin{lem} \label{lem:11}
Let $X$ and $Y$ be Banach spaces such that $Y \hookrightarrow X$ and $X^{*} \hookrightarrow Y^{*}$ with dense embedding, where $X^{*}$ and $Y^{*}$ denote the dual spaces of $X$ and $Y$, respectively. Then if a bounded sequence $\{\v_{n}\}_{n=1}^{\infty} \subset Y$ satisfies $\v_{n} \rightarrow 0$ in $X$ as $n \rightarrow \infty$, then for any $f \in Y^{*}$, $\langle \v_{n}, f \rangle_{Y \times Y^{*}} \rightarrow 0$ as $n \rightarrow \infty$. 
\end{lem}
Noting that for any $u \in B_{\p, 2}^{s}$, there exists $\{u_{j}\}_{j=1}^{\infty} \subset \mathscr{S}$ such that $u_{j} \rightarrow u\ in\ B_{\p, 2}^{s}$ as $j \rightarrow \infty$, and Lemma \ref{lem:11}, We obtain 
\begin{align}
\re \left\langle f(u(\wt{\ta})), \overline{\n u(\wt{\ta})} \right\rangle = \lim_{j \rightarrow \infty}\re \left\langle f(u_{j}(\wt{\ta})), \overline{\n u_{j}(\wt{\ta})} \right\rangle \ a.e.\wt{\ta}. \label{app:01}
\end{align}
In conclusion, by (\ref{app:01}) and $G(|u_{j}|) \in W^{1, 1}$ for all $j \in \N$, we get $\re \left\langle f(u(\wt{\ta})), \n u(\wt{\ta}) \right\rangle = 0\ a.e.s$, where $G(r) = \int_{0}^{r}f(\p) d\p$ for all $r \geq 0$. This completes the proof.
\end{proof}

\begin{proof}[Proof of Proposition \ref{pro:02}]
We can give the proof in a way similar to Ozawa \cite{bib002}. Acting $\n$ on (\ref{eq:04}), for all $t \in [T_{1}, T_{2}]$, we obtain  
\begin{align}
&||\n v(t)||_{L^{2}}^{2} \nonumber \\
&=||\n U(-t)v(t)||_{L^{2}}^{2} \nonumber \\
&=||\n \f||_{L^{2}}^{2}-2\im \left(U(-T_{1})\n \f, \int_{T_{1}}^{t}s^{\frac{n(p-1)-4}{2}}U(-s)\n f(v(s))ds \right)_{L^{2}} \nonumber \\
&\quad +\left\|\int_{T_{1}}^{t}s^{\frac{n(p-1)-4}{2}}U(-s)\n f(v(s))ds \right\|^{2}_{L^{2}}. \label{term:01} 
\end{align}
The second term on the RHS of (\ref{term:01}) satisfies the following equality: 
\begin{align}
&-2\im \left(U(-T_{1})\n \f, \int_{T_{1}}^{t}s^{\frac{n(p-1)-4}{2}}U(-s)\n f(v(s))ds \right)_{L^{2}} \nonumber \\
&\qquad  = -2\im\int_{T_{1}}^{t} \left\langle U(s-T_{1})\n \f, \overline{s^{\frac{n(p-1)-4}{2}} \n f(v(s))} \right\rangle ds, \label{term02}
\end{align}
where  combining Strichartz estimates with $\n f(v) \in L^{q'}_{t}L^{r'}_{x}$, the time integral of the scalar product is understood as the duality coupling on $(L_{t}^{q}L^{r}_{x})\times (L^{q'}_{t}L^{r'}_{x})$ with $(q,r)=(4(p+1)/(n(p-1)), p+1)$. For the last term on the RHS of (\ref{term:01}), Fubini's theorem implies
\begin{align}
&\left\|\int_{T_{1}}^{t}s^{\frac{n(p-1)-4}{2}}U(-s)\n f(v(s))ds \right\|^{2}_{L^{2}} \nonumber \\
&\qquad =2\re \int_{T_{1}}^{t}\left\langle s^{\frac{n(p-1)-4}{2}} \n f(v(s)), \overline{\int_{T_{1}}^{s}U(s-s') \n f(v(s'))ds'} \right\rangle ds, \label{term:03}
\end{align}
where the time integral of the scalar product is understood as the duality coupling on $(L^{q'}_{t}L^{r'}_{x}) \times (L_{t}^{q}L^{r}_{x})$. Concatenating (\ref{term:01}) - (\ref{term:03}), we compute
\begin{align*}
&\|\n v(t)\|_{L^{2}}^{2} \\
&=||\n \f||_{L^{2}}^{2}-2\im\int_{T_{1}}^{t}\left\langle U(s-T_{1})\n \f, \overline{s^{\frac{n(p-1)-4}{2}} \n f(v(s))} \right\rangle ds \\
&\qquad +2\re \int_{T_{1}}^{t}\left\langle s^{\frac{n(p-1)-4}{2}} \n f(v(s)), \overline{\int_{T_{1}}^{s}U(s-s') \n f(v(s'))ds'} \right\rangle ds \\
&=||\n \f||_{L^{2}}^{2}+2\im \int_{0}^{t} \left\langle s^{\frac{n(p-1)-4}{2}} \n f(v(s)), \overline{U(s-T_{1})\n \f} \right\rangle ds \\
&\qquad +2\im \int_{0}^{t}\left\langle s^{\frac{n(p-1)-4}{2}} \n f(v(s)), \overline{-i\int_{T_{1}}^{s}U(s-s') \n f(v(s'))ds'} \right\rangle ds  \\
&=||\n \f||_{L^{2}}^{2}+\lim_{\e \downarrow 0}2\im \int_{T_{1}}^{t}s^{\frac{n(p-1)-4}{2}} \left\langle (1-\e\D)^{-1} \n f(v(s)), \overline{\n v(s)} \right\rangle ds, 
\end{align*}
where the last equality in the above holds by using (\ref{eq:04}). Taking the duality coupling between the equation (\ref{eq:11}) and $(1-\e\D)^{-1}f(v)$ on $H^{-1}\times H^{1}$ and using $\im\{\langle (1-\e\D)^{-1}f(v),\overline{f(v)}\rangle \}=0$, we obtain 
\begin{align*}
\im \left\langle (1-\e\D)^{-1}\n f(v),\overline{\n v} \right\rangle &= \im\{-i \left\langle (1-\e\D)^{-1}f(v),\overline{\pa_{t}v} \right\rangle\}. 
\end{align*}
From these equalities, we can show
\begin{align}
&\|\n v(t)\|_{L^{2}}^{2} \nonumber \\
&=||\n \f||_{L^{2}}^{2}-\lim_{\e \downarrow 0}4\re \int_{0}^{t}s^{\frac{n(p-1)-4}{2}} \left\langle (1-\e\D)^{-1}f(v(s)), \overline{\pa_{t}v(s)} \right\rangle ds \nonumber \\
&=||\n \f||_{L^{2}}^{2}-4\re \int_{T_{1}}^{t}s^{\frac{n(p-1)-4}{2}} \left\langle f(v(s)), \overline{\pa_{t}v(s)} \right\rangle ds. \label{5eq:01}
\end{align}
Note that in the above, the time integral of the scalar product in the last line is understood as the duality coupling on $(L^{q'}_{t}W^{1,r'}_{x}) \times (L_{t}^{q}W^{-1, r}_{x})$. From (\ref{5eq:01}), we can continue as follows:
\begin{align*}
\|\n v(t)\|_{L^{2}}^{2} &=||\n \f||_{L^{2}}^{2}- 2\int_{T_{1}}^{t}s^{\frac{n(p-1)-4}{2}} \frac{d}{d s}\left( \frac{2\l}{p+1}\|v(s)\|_{L^{p+1}}^{p+1}\right) ds \\
&=||\n \f||_{L^{2}}^{2}- 2\int_{T_{1}}^{t}\frac{d}{d s} \left(s^{\frac{n(p-1)-4}{2}} \frac{2\l}{p+1}\|v(s)\|_{L^{p+1}}^{p+1}\right) ds \\
&\quad +2\int_{T_{1}}^{t}\frac{n(p-1)-4}{2} s^{\frac{n(p-1)-6}{2}} \frac{2\l}{p+1} \|v(s)\|_{L^{p+1}}^{p+1} ds \\
&=||\n \f||_{L^{2}}^{2} - 2t^{\frac{n(p-1)-4}{2}}\frac{2\l}{p+1}\|v(t)\|_{L^{p+1}}^{p+1} + 2T_{1}^{\frac{n(p-1)-4}{2}}\frac{2\l}{p+1}\|v(T_{1})\|_{L^{p+1}}^{p+1} \\
&\quad + \frac{2\l \{n(p-1)-4\}}{p+1}\int_{T_{1}}^{t}s^{\frac{n(p-1)-6}{2}}\|v(s)\|_{L^{p+1}}^{p+1} ds,
\end{align*}
since we can show that 
\begin{align}
 \frac{d}{d s}\left( \frac{2\l}{p+1}\|v(s)\|_{L^{p+1}}^{p+1}\right) = 2\re \langle f(v(s)), \overline{\pa_{t}v(s)}\rangle\ in\ L^{1}(T_{1}, T_{2}). \label{last:01}
\end{align}
We can justify the equality (\ref{last:01}) above by combining the way to the proof of Lemma 5.1 of \cite{bib005} with Lemma \ref{lem:11}. This completes the proof.
\end{proof}

\section*{Acknowledgment}
The authors would like to express deep gratitude to Professor Mishio Kawashita for helpful comments and warm encouragements. The authors are deeply grateful to Professor Tohru Ozawa for important comments. The authors would like to thank Dr. Yuta Wakasugi for constructive comments for the proof of Proposition \ref{pro:01}.
The authors would like to thank Professor Satoshi Masaki for helpful comments for how to derive the pseudo conformal conservation law.

\end{document}